\documentclass[10pt,a4paper]{amsart}
\usepackage[utf8]{inputenc}
\usepackage[T1]{fontenc}
\usepackage{amsmath}
\usepackage{upquote}
\usepackage{amsfonts}
\usepackage{amssymb}
\usepackage[english]{babel}
\usepackage{tabularx}
\usepackage{mathtools}
\usepackage{adjustbox}
\usepackage{graphics}
\usepackage{hyperref}
\usepackage{xcolor}
\usepackage{color}
\usepackage{tikz}
\usepackage{chngcntr}
\usetikzlibrary{calc}
\numberwithin{equation}{section}
\newtheorem{thm}{Theorem}[section]
\newtheorem{pr}[thm]{Proposition}

\newtheorem{ex}[thm]{Example}

\newtheoremstyle{case}{}{}{}{}{}{:}{ }{}
\theoremstyle{case}

\counterwithin*{case}{thm}
\newtheoremstyle{caso}{}{}{}{}{}{:}{ }{}
\theoremstyle{caso}

\DeclareRobustCommand{\svdots}{
  \vbox{%
    \baselineskip=0.33333\normalbaselineskip
    \lineskiplimit=0pt
    \hbox{.}\hbox{.}\hbox{.}%
    \kern-0.2\baselineskip
  }%
}

\newcommand{\N}{\mathbb{N}}

\newcommand{\R}{\mathbb{R}}

\theoremstyle{remark}

\makeatletter
\let\@@pmod\pmod
\DeclareRobustCommand{\pmod}{\@ifstar\@pmods\@@pmod}
\def\@pmods#1{\mkern4mu({\operator@font mod}\mkern 6mu#1)}
\makeatother

\title[On Bessenrodt-Ono type inequalities and log-concavity property]{On general approach to Bessenrodt-Ono type inequalities and log-concavity property}
\author{Krystian Gajdzica, Piotr Miska and Maciej Ulas}
\address{Institute of Mathematics \\
	Faculty of Mathematics and Computer Science \\
	Jagiellonian University in Cracow, Łojasiewicza 6, 30-348 Kraków, Poland
}
\email{krystian.gajdzica@doctoral.uj.edu.pl}
\email{piotr.miska@uj.edu.pl}
\email{maciej.ulas@uj.edu.pl}


\thanks{The research of the authors was supported by a grant of the National Science Centre (NCN), Poland, no. UMO-2019/34/E/ST1/00094}

\begin{document}

\maketitle

\begin{abstract}
In recent literature concerning integer partitions one can find many results related to both the Bessenrodt-Ono type inequalities and log-concavity property. In this note we offer some general approach to this type of problems. More precisely, we prove that under some mild conditions on an increasing function $F$ of at most exponential growth satisfying the condition $F(\N)\subset \R_{+}$, we have $F(a)F(b)>F(a+b)$ for sufficiently large positive integers $a, b$. Moreover, we show that if the sequence $(F(n))_{n\geq n_{0}}$ is log-concave and $\limsup_{n\rightarrow +\infty}F(n+n_{0})/F(n)<F(n_{0})$, then $F$ satisfies the Bessenrodt-Ono type inequality.

\end{abstract}

\section{Introduction}\label{sec1}
Let $\N$ be the set of non-negative integers, $\N_{+}$ the set of positive integers and $\R_{+}=(0,+\infty)$. Moreover, for an integer $n_{0}\geq2$ we put $N_{\geq n_{0}}=\{n\in\N:\;n\geq n_{0}\}$. 

For given $A\subset \N_{+}$ and $n\in\N$ by $p_{A}(n)$ we denote the number of partitions of $n$ with parts in the set $A$. As usual, we put $p_{A}(0)=1$. It is well known that the ordinary generating function for the sequence $(p_{A}(n))_{n\in\N}$ takes the form
$$
\sum_{n=0}^{\infty}p_{A}(n)x^{n}=\prod_{a\in A}\frac{1}{1-x^{a}}.
$$

If $A=\N_{+}$ then we simply write $p(n)$ instead of $p_{\N_{+}}(n)$. In this case the partition function is the famous Euler partition function thoroughly studied by Ramanujan and many others. The number of papers devoted to various properties of $p(n)$, or more generally, for $p_{A}(n)$ for various choice of the set $A$ is enormous. The standard reference is the book of Andrews \cite{GA2} (see also \cite{GA1} for less advanced approach).

A few years ago there emerged a broad research devoted to log-behavior of partition statistics. DeSalvo and Pak \cite{DSP} were the first who initiated this kind of study. More precisely, they investigated the log-concavity of the partition function $p(n)$. Let us recall that a sequence $(\omega_n)_{n\in\N}$ of real numbers is said to be log-concave if the inequality
$$
\omega_n^2> \omega_{n-1}\omega_{n+1}
$$ 
is valid for all sufficiently large values of $n$. On the other hand, it is called log-convex if the inverse inequality holds.

Using analytic methods DeSalvo and Pak \cite{DSP} reproved the Nicolas' theorem \cite{Nicolas} which states that the sequence $(p(n))_{n\geq 26}$ is log-concave.

Now, there is a wealth of literature devoted to the log-concavity property for other variations of the partition function. Since we are unable to present a complete list of the papers concerning the topic, let us recall at least a few of them. For instance, Bringmann, Kane, Rolen and Tripp \cite{BKRT} investigated the case of the $k$-colored partition function. On the other hand, Dawsey and Masri \cite{DM} examined the Andrews {\it spt}-function in that direction. Further, Engel \cite{Engel} proved that the overpartition function $\overline{p}(n)$ is log-concave for every $n\geq2$. Gajdzica \cite{Gajdzica} discovered a similar phenomenon for the $A$-partition function when $A$ is finite. O’Sullivan \cite{Sullivan} investigated the number of partitions into powers and proved a conjecture of Ulas \cite{U}. One can also mention that Ono, Pujahari and Rolen \cite{OPR} showed that the plane partition function satisfies the log-concavity property as well.

However, the research devoted to log-behavior of partition functions is not only bounded by the log-concavity or log-convexity properties. Another interesting phenomenon is the so-called Bessenrodt-Ono inequality. More precisely, in 2016, Bessenrodt and Ono \cite{B-O} showed that for $a, b\geq 2$ and $a+b>9$ we have
\begin{equation}\label{Theorem: Bessenrodt-Ono (1)}
p(a)p(b)>p(a+b).
\end{equation}

Their proof is based on the asymptotic estimates due to Lehmer \cite{L}. It is worth noting that there are two alternative approaches to derive the result. Alanazi, Gagola and Munagi \cite{AGM} showed the Bessenrodt-Ono inequality in combinatorial manner by constructing appropriate injections between some sets of partitions. Heim and Neuhauser \cite{HN3}, on the other hand, presented a proof which is based on the induction on $a+b$.

Moreover, the inequality \eqref{Theorem: Bessenrodt-Ono (1)} can be effectively used in practice. For instance, one can apply it to determine the value of
$$\max{p(n)}=\max{\left\{p(\boldsymbol{\lambda}):\boldsymbol{\lambda} \text{ is a partition of } n\right\}},$$
where $p(\boldsymbol{\lambda})$ denotes the extended partition function defined as $p(\boldsymbol{\lambda}):=\prod_{i=1}^jp(\lambda_i)$ for $\bf{\lambda}=(\lambda_1,\lambda_2,\ldots,\lambda_j)$. For additional information, we refer the reader to \cite{B-O}.

A lot of works concerning the Bessenrodt-Ono type inequality were published. For example, Beckwith and Bessenrodt \cite{Beckwith} found out similar properties for the $k$-regular partition function. Chern, Fu and Tang \cite{CFT} considered the $k$-colored partition function. Dawsey and Masri \cite{DM} discovered the analogous phenomenon for the Andrews {\it spt}-function. Heim, Neuhauser and Tr\"oger  \cite{HNT} examined the issue for the plane partition function. Moreover, Heim and Neuhauser widely generalized the property and investigated the so-called polynomization of the Bessenrodt-Ono inequality for a couple of partition functions \cite{KG3, HN4, HN6, HNT}. Hou and Jagadeesan \cite{HJ}, and Males \cite{M} discovered the analogues of the Bessenrodt-Ono inequality for the so-called partition rank function.

Actually, there are many more properties related to the log-behavior of partition statistics. We do not discuss them here, but focus on some general criteria for both Bessenrodt-Ono type inequalities and log-concavity problems. It turns out that\linebreak a plenty of the aforementioned results are proved using some asymptotic estimates. Therefore, it would be convenient to posses a few results which assert that a function with an appropriate growth is log-concave or fulfills the Bessenrodt-Ono inequality. Essentially, these goals were the main motivations for our investigation.

Let us describe the content of the paper in some details. In Section \ref{sec2} we get a general criterion on the sequence $(F(n))_{n\in\N}$, which guarantees that it satisfies the Bessenrodt-Ono type inequality. In other words, for all sufficiently large $a, b\in\N, a\geq b$ we have $F(a)F(b)>F(a+b)$. In particular, as an application we reprove recent result of Heim and Neuhauser which says that the plane partition function satisfies the Bessenrodt-Ono type inequality. In Section \ref{sec3} we present some conditions guaranteeing asymptotic log-concavity of the sequence $(F(n))_{n\in\N}$. As an application we reprove recent result of DeSalvo and Pak \cite[Theorem 1.1]{DSP}. Finally, in the last section, under some mild condition on the sequence $(F(n))_{n\in\mathbb{N}}$ we show that log-concavity implies the Bessenrodt-Ono type inequality. We also prove that the opposite implication is not true. More precisely,  we show that for each $m\in\N_{\geq 2}$ the $m$-ary partition function $b_{m}(n)=p_{A}(n)$, where $A=\{m^{i}:\;i\in\N\}$, satisfies the Bessenrodt-Ono type inequality but is not log-concave.

\section{Bessenrodt-Ono type inequality holds for a class of sub-exponential functions}\label{sec2}

In this section we offer a general approach to Bessenrodt-Ono type inequalities. More precisely, we prove the following general result.

\begin{thm}\label{theorem: B-O (2) modified}
    Let $c_1,c_2,f,F:\mathbb{N}_+\to\mathbb{R}_+$ be given and such that the inequalities
    \begin{align*}
        c_1(n)e^{f(n)}<F(n)<c_2(n)e^{f(n)}
    \end{align*}
    are valid for all positive integers $n\geq N_{0}$ for some $N_0\in\mathbb{N}_+$. Suppose further that the following conditions hold:
    \begin{enumerate}
        \item $\exists g:\mathbb{N}_+\to\mathbb{R}_+, N_1\in\mathbb{N}_+\text{ }\forall a\geq b\geq N_1:f(a)+f(b)-f(a+b)\geq g(b)$;
        \item $\exists  h:\mathbb{N}_+\to\mathbb{R}_+,N_2\in\mathbb{N}_+\text{ }\forall a\geq b\geq N_2: c_2(a+b)/c_1(a)\leq h(b)$;
        \item $\exists N_3\in \mathbb{N}_+\text{ }\forall n\geq N_3: g(n)\geq \log h(n)-\log c_1(n)$.
    \end{enumerate}
    Then, the inequality
    \begin{align*}
        F(a)F(b)>F(a+b)
    \end{align*}
    is satisfied for all $a,b\geq \max\{N_0, N_1, N_2, N_{3}\}$.
\end{thm}
\begin{proof}
    Let us fix functions $c_1,c_2,f,F$ and $g$ as in the statement. Since both of the inequalities
    \begin{align*}
        F(a)F(b)>c_1(a)c_1(b)e^{f(a)+f(b)}
    \end{align*}
    and
    \begin{align*}
        F(a+b)<c_2(a+b)e^{f(a+b)}
    \end{align*}
    are true for any positive integers $a,b\geq N_{0}$, it is enough to prove that the following inequality
    \begin{align*}
        e^{f(a)+f(b)-f(a+b)}\geq \frac{c_2(a+b)}{c_1(a)c_1(b)}
    \end{align*}
    is satisfied for all sufficiently large values of $a$ and $b$. From ($1$), it follows that the inequality
    \begin{align*}
        f(a)+f(b)-f(a+b)\geq g(b)
    \end{align*}
    holds for every $a\geq b\geq N_1$. On the other hand, ($2$) implies that the inequality
    \begin{align*}
        \frac{c_2(a+b)}{c_1(a)}\leq h(b)
    \end{align*}
    is valid for any $a\geq b\geq N_2$. Hence, it is enough to show that the following
    \begin{align*}
        e^{g(b)}\geq h(b)/c_1(b)
    \end{align*}
    holds for all large values of $b$ --- but that is a direct consequence of ($3$), as required.
\end{proof}

Although the above theorem is very easy we show that it is strong enough to be applied to classical examples of partition functions.

\begin{ex}\label{example1}
{\rm
    Let us consider the partition function $p(n)$. It follows from the Bessenrodt and Ono's paper \cite{B-O} that the inequalities
    \begin{align*}
        \frac{\sqrt{3}}{12n}\left(1-\frac{1}{\sqrt{n}}\right)e^{\frac{\pi}{6}\sqrt{24n-1}}<p(n)<\frac{\sqrt{3}}{12n}\left(1+\frac{1}{\sqrt{n}}\right)e^{\frac{\pi}{6}\sqrt{24n-1}}
    \end{align*}
    hold for every positive integer $n$. Therefore, we will use Theorem \ref{theorem: B-O (2) modified} with $F(n):=p(n)$,
    \begin{align*}
        f(n):=\frac{\pi}{6}\sqrt{24n-1},\hspace{0.3cm} c_1(n):=\frac{\sqrt{3}}{12n}\left(1-\frac{1}{\sqrt{n}}\right)\hspace{0.3cm} \text{and}\hspace{0.3cm}c_2(n):=\frac{\sqrt{3}}{12n}\left(1+\frac{1}{\sqrt{n}}\right).
    \end{align*}
Now, if we assume that $a\geq b\geq 1$, then we have that
\begin{align*}
    f(a)+f(b)-f(a+b)&=\frac{\pi}{6}\cdot\frac{2\sqrt{24a-1}\sqrt{24b-1}-1}{\sqrt{24a-1}+\sqrt{24b-1}+\sqrt{24(a+b)-1}}\\
    &>\frac{\pi}{6}\cdot\frac{\sqrt{24a-1}\sqrt{24b-1}}{\sqrt{24a-1}+\sqrt{24b-1}}-\frac{1}{24}\\
    &\geq\frac{\pi}{12}\sqrt{24b-1}-\frac{1}{24}=:g(b)
\end{align*}
    On the other hand, one can derive that
    \begin{align*}
        \frac{c_2(a+b)}{c_1(a)}&=\frac{12a}{12(a+b)}\left(1+\frac{1}{\sqrt{a+b}}\right)\left(1+\frac{1}{\sqrt{a}-1}\right)<1+\frac{2}{\sqrt{a}-1}\leq2=:h(b)
    \end{align*}
    whenever $a\geq9$ and $b\leq a$. Hence, it is enough to observe that
    \begin{align*}
        e^{g(n)}\geq 8\sqrt{3}n\left(1+\frac{1}{\sqrt{n}-1}\right)
    \end{align*}
    is true for all sufficiently large values of $n$. In fact, one can show that the above is valid for every $n\geq 15$. Thus, Theorem \ref{theorem: B-O (2) modified} implies that the Bessenrodt-Ono inequality
    \begin{align*}
        p(a)p(b)>p(a+b)
    \end{align*}
    holds for all $a\geq b\geq15$.
    }
\end{ex}

\begin{ex}
{\rm   Let us recall that the number $pp(n)$ of plane partitions of $n$  can be computed via the generating function (obtained by MacMahon \cite{MacMa})
$$
\sum_{n=0}^{\infty}pp(n)x^{n}=\prod_{n=1}^{\infty}\frac{1}{(1-x^{n})^{n}}.
$$

From Wright's formula \cite{Wright} one can deduce the existence of $\alpha,\beta,\gamma,N>0$ such that for all $n>N$ the following inequalities
    \begin{align*}
        \alpha n^{-\frac{25}{36}}\left(1-\frac{\beta}{\sqrt{n}}\right)e^{\gamma n^{2/3}}<pp(n)<\alpha n^{-\frac{25}{36}}\left(1+\frac{\beta}{\sqrt{n}}\right)e^{\gamma n^{2/3}}
    \end{align*}
    hold. We set $F(n):=pp(n)$, 
    \begin{align*}
        f(n):=\gamma n^{2/3},\hspace{0.3cm} c_1(n):=\alpha n^{-\frac{25}{36}}\left(1-\frac{\beta}{\sqrt{n}}\right)\hspace{0.3cm} \text{and}\hspace{0.3cm}c_2(n):=\alpha n^{-\frac{25}{36}}\left(1+\frac{\beta}{\sqrt{n}}\right),
    \end{align*}
    and apply Theorem \ref{theorem: B-O (2) modified}. For every $a\geq b\geq N$, we have that
    \begin{align*}
        f(a)+f(b)-f(a+b)&=\gamma\left(a^{\frac{2}{3}}+b^{\frac{2}{3}}-a^{\frac{2}{3}}\left(1+\frac{b}{a}\right)^{\frac{2}{3}}\right)\\
        &\geq \gamma\left(a^{\frac{2}{3}}+b^{\frac{2}{3}}-a^{\frac{2}{3}}\left(1+\frac{2b}{3a}\right)\right)=\gamma\left(b^\frac{2}{3}-\frac{2b}{3a^{1/3}}\right)\geq\frac{\gamma}{3}b^\frac{2}{3},
    \end{align*}
    where the first inequality is a consequence of Bernoulli's inequality. On the other hand, it is not difficult to see that
    \begin{align*}
        \frac{c_2(a+b)}{c_1(a)}=\left(1+\frac{b}{a}\right)^{-\frac{25}{36}}\left(1+\frac{\beta}{\sqrt{a+b}}\right)\left(1+\frac{\beta}{\sqrt{a}-\beta}\right)<1+\frac{2\beta}{\sqrt{a}-\beta}\leq2\beta+1
    \end{align*}
    whenever $a\geq b\geq \max\{N,(\beta+1)^2\}$. Now, it is straightforward to deduce that
    \begin{align*}
        e^{\frac{\gamma}{3}n^\frac{2}{3}}\geq\frac{2\beta+1}{\alpha} n^{\frac{25}{36}}\left(1+\frac{\beta}{\sqrt{n}-\beta}\right)
    \end{align*}
    is satisfied for all but finitely many values of $n$. In conclusion, we get that $pp(n)$ fulfills the Bessenrodt-Ono type inequality for  all large parameters $a$ and $b$.
    }
\end{ex}

\section{Log-concavity property holds for a class of sub-exponential functions}\label{sec3}

Let us recall that a sequence $(F(n))_{n\in\N}$ (or just a function $F:\;\N\rightarrow \R_{+}$) is log-concave if
$$
F(n)^2> F(n-1)F(n+1)
$$
for all $n\geq n_0$, where $n_{0}$ is some positive integer. 

In this section we get a general result which under mild conditions on the growth of the function $F:\mathbb{N}\to\mathbb{R}_{+}$ guarantees that $F$ is log-concave.

\begin{thm}\label{theorem: log-concavity}
    Let $c_1,c_2,f,F:\mathbb{N}\to\mathbb{R}_+$ be given and such that the inequalities
    \begin{align*}
        c_1(n)e^{f(n)}<F(n)<c_2(n)e^{f(n)}
    \end{align*}
    are valid for all positive integers $n\geq N_{0}$ for some $N_0\in\mathbb{N}$. Suppose further that the following conditions hold:
    \begin{enumerate}
    \item $\exists N_1\in\mathbb{N}_+\text{ }\exists h:\mathbb{N}_+\to\mathbb{R}_{+}\text{ }\forall n\geq N_1: h(n)\leq2f(n)-f(n-1)-f(n+1)$;
    \item $\exists N_2\in\mathbb{N}_{+}\text{ }\forall n\geq N_2: c_2(n+1)c_2(n-1)/c_1^2(n)\leq e^{h(n)}$.

    \end{enumerate}
    Then, for all the values of $n\geq\max \{N_{0}, N_{1}, N_{2}\}$ the inequality
    \begin{align*}
        F^2(n)>F(n-1)F(n+1)
    \end{align*}
    is true.
\end{thm}
\begin{proof}
Let $c_1,c_2,f,F,h$ be as in the statement of the theorem. It follows that both of the inequalities
\begin{align*}
        F^2(n)>c_1^2(n)e^{2f(n)}
\end{align*}
and
    \begin{align*}
        F(n-1)F(n+1)<c_2(n-1)c_2(n+1)e^{f(n-1)+f(n+1)}
    \end{align*}
are valid for every $n\geq N_{0}$. Therefore, it is enough to show that the inequality
$$
        e^{h(n)}\geq\frac{c_2(n-1)c_2(n+1)}{c_1^2(n)}
$$
is true for all sufficiently large values of $n$, but this is exactly $(2)$. Hence, we conclude that the inequality
$$
        F^2(n)>F(n-1)F(n+1)
$$
is satisfied for all $n\geq\max\{N_{0}, N_1, N_2\}$, as required.
\end{proof}

\begin{ex}
{\rm For every positive integer $n\geq 37$, we have that
\begin{align}\label{Chen's inequalities for p(n)}
    c_1(n)e^{\mu(n)}<p(n)<c_2(n)e^{\mu(n)},
\end{align}
where
\begin{align*}
    \mu(n)&:=\frac{\pi}{6}\sqrt{24n-1},\\
    c_1(n)&:=\frac{\sqrt{12}}{24n-1}\left(1-\frac{1}{\mu(n)}-\frac{1}{\mu^3(n)}\right),\\
    c_2(n)&:=\frac{\sqrt{12}}{24n-1}\left(1-\frac{1}{\mu(n)}+\frac{1}{\mu^3(n)}\right).
\end{align*}
    The inequalities ($\ref{Chen's inequalities for p(n)}$) follow directly from Chen, Jia and Wang \cite[Lemma 2.2]{Chen} and some numerical computations carried out in Wolfram Mathematica \cite{WM}.

Now, we want to apply Theorem \ref{theorem: log-concavity} to deduce the log-concavity for $p(n)$. At first, let us observe that the generalized binomial theorem asserts that for every $|j|<|n|$ the following inequalities
    \begin{align}\label{Generalized binomial for u(n)}
       t_{-j}(n)<\sqrt{n+j}<t_{+j}(n)
    \end{align}
are true, where
\begin{align*}
    t_{-j}(n)=n^\frac{1}{2}+\frac{1}{2}jn^{-\frac{1}{2}}-\frac{1}{8}j^2n^{-\frac{3}{2}}-2\cdot|j|^3n^{-\frac{5}{2}}
\end{align*}
and
\begin{align*}
    t_{+j}(n)=n^\frac{1}{2}+\frac{1}{2}jn^{-\frac{1}{2}}-\frac{1}{8}j^2n^{-\frac{3}{2}}+2\cdot|j|^3n^{-\frac{5}{2}}.
\end{align*}
    Therefore, we have that
    \begin{align*}
        2\mu(n)-\mu(n-1)-\mu(n+1)&=\frac{\sqrt{24}\pi}{6}\left(2\sqrt{n-\frac{1}{24}}-\sqrt{n-\frac{25}{24}}-\sqrt{n+\frac{23}{24}}\right)\\
        &\geq \frac{\sqrt{24}\pi}{6}\left(\frac{1}{4}n^{-\frac{3}{2}}-\frac{55588}{13824}n^{-\frac{5}{2}}\right).
    \end{align*}
    On the other hand, one can also calculate that
    \begin{align*}
        \frac{c_2(n-1)c_2(n+1)}{c_1^2(n)}&=\left(1+\frac{24^2}{(24n-1)^2-24^2}\right)^\frac{5}{2}\\
        &\times\frac{\left(\mu^3(n-1)-\mu^2(n-1)+1\right)\left(\mu^3(n+1)-\mu^2(n+1)+1\right)}{\left(\mu^3(n)-\mu^2(n)-1\right)^2}\\
        &\leq \left(1+\frac{24^2}{(24n-1)^2-24^2}\right)^\frac{5}{2}\times\left(1+\frac{24\sqrt{6}}{7\pi^3}n^{-\frac{3}{2}}\right),
    \end{align*}
    where the last inequality is a consequence of ($\ref{Generalized binomial for u(n)}$) and some elementary but tiresome computations. Hence, it is enough to check under what conditions on $n$ the inequality
    \begin{align*}
        1+\frac{\sqrt{24}\pi}{6}\left(\frac{1}{4}n^{-\frac{3}{2}}-\frac{55588}{13824}n^{-\frac{5}{2}}\right)\geq\left(1+\frac{24^2}{(24n-1)^2-24^2}\right)^\frac{5}{2}\times\left(1+\frac{24\sqrt{6}}{7\pi^3}n^{-\frac{3}{2}}\right)
    \end{align*}
    is true. One can verify that it holds for all $n\geq94$. Therefore, if we examine the positivity of $p^2(n)-p(n-1)p(n+1)$ for every $1\leq n\leq 93$, then we obtain \cite[Theorem 1.1]{DSP}.
    }
\end{ex}

\section{Log-concavity (usually) implies Bessendrodt-Ono type inequality}\label{sec4}

In this section we show a strict connection between log-concavity property and the Bessenrodt-Ono type inequality.

\begin{thm}\label{proposition: log-concavity => B-O}
Let $f:\mathbb{N}_+\to\mathbb{R}_+$ be fixed. Assume further that there exists $n_0\in\N_+$ such that $(f(n))_{n\geq n_0}$ is log-concave and
\begin{align*}
    \limsup_{n\to\infty}\frac{f(n+n_0)}{f(n)}<f(n_0).
\end{align*}
Then, the inequality
\begin{align*}
    f(a)f(b)>f(a+b)
\end{align*}
holds for all sufficiently large numbers $a, b\in\N_+$.
\end{thm}
\begin{proof}
    At first, let us observe that the log-concavity property implies that the inequality
    \begin{align}
        \frac{f(n)}{f(n-1)}>\frac{f(n+1)}{f(n)}\label{equivalence of log-concavity}
    \end{align}
    is valid for each $n\geq n_0$. Thus, it is enough to check the validity of the inequality in
    \begin{align*}
        f(a+b)&=\frac{f(a+b)}{f(a+b-1)}\times\cdots\times\frac{f(a+n_0+1)}{f(a+n_0)}f(a+n_0)\\
        &<f(a)\frac{f(b)}{f(b-1)}\times\cdots\times\frac{f(n_0+1)}{f(n_0)}f(n_0)=f(a)f(b)
    \end{align*}
    for all sufficiently large $a$ and $b$. However, \eqref{equivalence of log-concavity} points out that
    \begin{align*}
        \frac{f(n_0+i+1)}{f(n_0+i)}>\frac{f(a+n_0+i+1)}{f(a+n_0+i)}
    \end{align*}
    is true for any $i=0,1,\ldots,b-n_0-1$. Therefore, the task boils down to proving that the inequality
    \begin{align*}
        \frac{f(a+n_0)}{f(a)}<f(n_0)
    \end{align*}
    is true for all large parameters $a$ and $b$, but it is a direct consequence of the leftover assumption from the statement. This ends the proof.
\end{proof}

One can easily notice that besides Theorem \ref{proposition: log-concavity => B-O}, the above proof implies the following property.

\begin{pr}\label{corollary: log-concave => B-O}
    Let $f:\mathbb{N}_+\to\mathbb{R}_+$ be fixed. Assume further that there exists $n_0\in\N_+$ such that $(f(n))_{n\geq n_0}$ is log-concave and the inequality
\begin{align*}
    f(n)f(n_0)>f(n+n_0)
\end{align*}
is valid for every $n\geq n_0$. Then, we have that
\begin{align*}
    f(a)f(b)>f(a+b)
\end{align*}
is satisfied for all $a,b\geq n_0$.
\end{pr}
\begin{ex}
    Suppose that we know that the classical partition function satisfies
    \begin{align*}
        p(26)p(a)>p(a+26)
    \end{align*}
    for every $a\geq 26$. Then, from \cite[Theorem 1.1]{DSP} and Proposition \ref{corollary: log-concave => B-O} we get that the inequality
    \begin{align*}
        p(a)p(b)>p(a+b)
    \end{align*}
    holds for all $a,b\geq26$.
\end{ex}

\begin{thm}\label{Proposition: log-concavity => B-0 (2)}
    Let $f:\mathbb{N}\to\mathbb{R}_+$ be fixed and such that $f(0)\geq1$. Suppose further that the sequence $\left(f(n)\right)_{n=0}^\infty$ is log-concave for every positive integer $n$. Then
\begin{align*}
    f(a)f(b)>f(a+b)
\end{align*}
is true for all positive integers $a$ and $b$.
\end{thm}
\begin{proof}
    The log-concavity property and the assumption that $f(0)\geq1$ imply that
    \begin{align*}
        f(a+b)&=\frac{f(a+b)}{f(a+b-1)}\times\frac{f(a+b-1)}{f(a+b-2)}\times\cdots\times\frac{f(a+1)}{f(a)}f(a)\\
        &<f(a)\frac{f(b)}{f(b-1)}\times\frac{f(b-1)}{f(b-2)}\times\cdots\times\frac{f(1)}{f(0)}f(0)=f(a)f(b)
    \end{align*}
    holds for all positive integers $a$ and $b$. This completes the proof.
\end{proof}
\begin{ex}
{\rm    Since the sequence $(p(n))_{n\geq26}$ is log-concave by \cite[Theorem 1.1]{DSP} and we have $p(26)=2436>1$, Theorem \ref{Proposition: log-concavity => B-0 (2)} implies that the sequence $(p(n+26))_{n\in\N}$ satisfies the Bessenrodt-Ono type inequality. 
    }
\end{ex}

\begin{ex}
{\rm Let us set $q(n):=F_{2n}$, where $F_j$ denotes the $j$-th Fibonacci number.  From Cassini identity we deduce that
$$
q^2(n)-q(n+1)q(n-1)=F_{2n}^2-F_{2n-2}F_{2n+2}=1>0.
$$
Thus, the sequence $(q(n))_{n\in\N}$ is log-concave. However, it does not satisfy the Bessenrodt-Ono type inequality. Indeed, one can easily check that $\phi^{2n}/\sqrt{5}-1<F_{2n}<\phi^{2n}/\sqrt{5}$, where $\phi=(1+\sqrt{5})/2$ is the golden mean. Thus, for $a+b\geq 3$ we have
$$
q(a+b)-q(a)q(b)=F_{2(a+b)}-F_{2a}F_{2b}>\frac{\sqrt{5}-1}{5}\phi^{2(a+b)}-1>0.
$$
As a consequence, despite the fact that the sequence $(q(n))_{n\in\N}$ is log-concave, it does not satisfy the Bessenrodt-Ono type inequality. In conclusion, we see that the assumption $f(0)\geq 1$ in Proposition \ref{Proposition: log-concavity => B-0 (2)} is necessary.

On the other hand, let us note that for any fixed $j\in\N_{+}$ the sequence $(q(n+j))_{n\in\N}$ satisfies the Bessenrodt-Ono type inequality.
 


}
\end{ex}

Theorem \ref{proposition: log-concavity => B-O} asserts that a positive log-concave sequence usually satisfies the Bessenrodt-Ono inequality. Therefore, there appears a natural question whether the inverse statement is true. In general, it is not the case as the following example shows.

\begin{ex}
{\rm    For an arbitrary positive integer $m\geq2$, the $m$-ary partition function $b_m(n)$ is just $p_{A}(n)$ for $A=\{m^{i}:\;i\in\N\}$, i.e., $b_{m}(n)$ is the number of representations of $n$ as sums of powers of $m$. Let us recall that $b_{m}(0)=1, b_{m}(mn+i)=b_{m}(mn)$ for $i=1, \ldots, m-1$ and $b_{m}(mn)=b_{m}(mn-1)+b_{m}(n)$. 

The sequence $(b_m(n))_{n\in\N}$ is not log-concave. Indeed, if we have that $n\equiv-1\pmod*{m}$, then $b_m(n-1)=b_m(n)<b_m(n+1)$.

On the other hand, the following is true.

\begin{thm}
 The sequence $(b_m(n))_{n\in\N}$ satisfies the Bessenrodt-Ono type inequality for every $m\geq2$.
\end{thm}
\begin{proof}
In order to get the statement we apply Mahler's theorem (see \cite{De Bruijn, Mahler, Pennington}), which states that
    \begin{align*}
        \log b_m(n)\sim\frac{\left(\log n\right)^2}{2\log m}.
    \end{align*}
    In particular, it means that there are some functions $c_1,c_2:\mathbb{N}_+\to\mathbb{R}$ such that
    \begin{align*}
        \lim_{n\to\infty}c_1(n)=\lim_{n\to\infty}c_2(n)=1,
    \end{align*}
    and the inequalities
    \begin{align*}
        c_1(n)e^\frac{\left(\log n\right)^2}{2\log m}<b_m(n)<c_2(n)e^\frac{\left(\log n\right)^2}{2\log m}
    \end{align*}
    hold for all $n\geq N$ for some positive integer $N$.

To get the result we apply Theorem \ref{theorem: B-O (2) modified}. Let us put $f(n):=(\log n)^2/(2\log m)$. Without loss of generality we assume that $a\geq b$. Then we set $a=\lambda b$ for some $\lambda\geq 1$. We have that
\begin{align*}
        \log^2 a +\log^2 b-\log^2 (a+b)&= \log^2 (\lambda b)+\log^2 b-\log^2((\lambda+1)b)\\
        &=\log^2 b+2(\log b)\log\frac{\lambda}{\lambda+1}+\log^2 \lambda-\log^2 (\lambda+1)\\
        &>\left(\log b-2\log 2\right)\log b-\frac{3}{2},
\end{align*}
where the last inequality follows from elementary calculations. Therefore, one can put
    \begin{align*}
        g(n):=\frac{1}{2\log m}\left(\left(\log n-2\log 2\right)\log n-\frac{3}{2}\right).
    \end{align*}

    To find out an appropriate function $h(n)$ from Theorem \ref{theorem: B-O (2) modified}, we observe that for some $M\in\N$ and all $n\geq M$, the inequalities
    \begin{align*}
        c_1(n)\geq \frac{1}{2} \hspace{0.2cm}\text{and}\hspace{0.2cm}c_2(n)\leq2
    \end{align*}
    are true. Hence, if $a\geq b\geq M$, then
    \begin{align*}
        \frac{c_2(a+b)}{c_1(a)}\leq 4.
    \end{align*}
    Now, it is clear that the following
    \begin{align*}
        \frac{1}{2\log m}\left(\left(\log n-2\log 2\right)\log n-\frac{3}{2}\right)\geq 3\log2\geq \log4-\log c_1(n)
    \end{align*}
    is valid for all but finitely many values of $n$. In consequence, Theorem \ref{theorem: B-O (2) modified} guarantees that the sequence $(b_m(n))_{n\in\N}$ satisfies the Bessenrodt-Ono type inequality, as required.
\end{proof}
}
\end{ex}

\section*{Acknowledgments}
The first author would like to thank Bernhard Heim for fruitful discussion and valuable suggestions. Moreover, he is also supported by a grant from the Faculty of Mathematics and Computer Science under the Strategic Program Excellence Initiative at the Jagiellonian University in Kraków.

\end{document}